\RequirePackage{fix-cm}
\documentclass[5pt]{article}
\usepackage{amsfonts}
\usepackage{mathrsfs}
\usepackage{amsmath}
\usepackage{amsthm}
\usepackage{enumerate}
\usepackage{multicol}
\usepackage{lipsum}
\makeatletter
\def\ps@pprintTitle{%
   \let\@oddhead\@empty
   \let\@evenhead\@empty
   \def\@oddfoot{\reset@font\hfil\thepage\hfil}
   \let\@evenfoot\@oddfoot
}

\usepackage{graphicx}

\usepackage[top=3cm, bottom=3cm, left=3cm, right=2cm] {geometry}

\usepackage{hyperref}
\usepackage{fontenc}
\usepackage[v2,cmtip]{xy}
\usepackage{amsmath}
\usepackage{bm}
\usepackage{amssymb}
\usepackage{mathrsfs}
\usepackage{latexsym}
\usepackage{exscale}
\usepackage{eucal}
\vspace{2ex}
\title{\flushleft {A proof of cases of de Polignac's conjecture}}\vspace{2ex}

\author{Mbakiso F. Mothebe$^1,$  Dintle N. Kagiso$^2$  and Ben T. Modise$^3$ \\
\normalsize \it Department of Mathematics, University of Botswana, Pvt Bag 00704, \\
 \it Gaborone, Botswana$^{1,3}$\\
 \normalsize \it Botswana International University of Science and Technology, \\
 \normalsize \it Private Bag 16, Palapye, Botswana$^2$ \\
  \normalsize {\it e-mail add}: mothebemf@ub.ac.bw$^1$;
   dintlek.kagiso@gmail.com$^2$; modisesan@gmail.com$^3$ }
\date{}
\theoremstyle{plain}
\newtheorem{thm}{Theorem}[section]
\newtheorem{theorem}[thm]{Theorem}

\newtheorem{lemma}[thm]{Lemma}

\newtheorem{corollary}[thm]{Corollary}
\theoremstyle{definition}

\theoremstyle{plain}

\theoremstyle{definition}

\def\leq{\leqslant}
\def\geq{\geqslant}
\def\DD{D\kern-.7em\raise0.4ex\hbox{\char '55}\kern.33em}

\newtheorem*{acknow}{Acknowledgment}
\newtheorem*{intere}{Competing Interest}

\makeatletter
\def\blfootnote{\xdef\@thefnmark{}\@footnotetext}
\makeatother

\begin{document}
\maketitle
\pagestyle{plain}

\begin{abstract}
For $n \geq 1$ let $ p_n $ denote the $n^{\rm th}$ prime number. Let $$S= \{1,7,11,13,17,19,23,29 \},$$ the set of positive integers
 which are both less than and relatively prime to $30.$ For $ x \geq 0,$ let \\ $T_x := \{ 30x+i \; | \; i \in S\}.$ For each $ x,$  $T_x$ contains at most seven primes.
   Let  $[ \; ]$  denote the floor or greatest integer function.
  For each integer $s \geq 30$ let $\pi_7(s)$ denote
 the number of integers $x, \; 0 \leq x < [\frac {s}{30}]$ for which $T_x$ contains seven primes. Let $m \geq 10^{10}$ be an integer and let $P_{K_m}$ denote the largest prime number less than $\sqrt{\prod_{i=1}^{m}p_i}.$
 In this paper we show that
$$\frac{\prod_{i=1}^{m}p_i}{8(K_m+1)} < \pi_7\left(\prod_{i=1}^{m}p_i\right) $$ and thereby prove that
  there are infinitely
 many values of $x$ for which $T_x$ contains seven primes. This, in particular, proves the well known twin prime conjecture as well as several cases of Alphonse de Polignac's conjecture that for every even number $k,$ there are infinitely many pairs of prime numbers
    $p$ and $p'$ for which $p'-p = k.$

\end{abstract}

\section{Introduction and main result}\label{s1}
\setcounter{equation}{0}

An integer $p \geq 2$ is called a prime if its only positive
divisors are $1$ and $p.$ The prime numbers form a sequence:
\begin{equation}  \label{primes}
2,\; 3,\;5, \;7, \; 11, \; 13, \; 17, \; 19, \; 23, \; 29, \; 31, \;
37, \; 41,\; 43,\; \ldots.
\end{equation}
Euclid (300 B.C.) considered prime numbers and proved that there are infinitely
many.

 In $1849,$ Alphonse de Polignac \cite{Po} conjectured that for
 every even number $k,$ there are infinitely many pairs of prime numbers $p$ and $p'$ such that $p'-p = k.$ The case $k=2$ is the well known twin prime conjecture, which is proved in \cite{moth1}.
 The
 conjecture has not yet been proven or disproven for any given value of $k \neq 2.$ In $2013$ an important breakthrough was made by Yitang Zhang who proved the conjecture for
  some value of $k< 70 \; 000 \; 000$ \cite{zhang}.  Later that same year, James Maynard announced a related breakthrough which proved the conjecture for some $k< 600$ (see \cite{mayn}). In 2014 the D.H.J. Polymath project proved the conjecture for some $k \leq 246.$ (see \cite{poly})

In this paper we prove cases of de Polignac's conjecture which are implied by the following result. Our arguments are an extension or generalization of the arguments developed in \cite{moth1}.

     \vspace{3ex}
\noindent \textbf{AMS Subject Classification:} 11N05; 11N36.

\vspace{.08in} \noindent \textbf{Keywords}: Primes, Seven prime sets,
Sieve methods.

\vspace{3ex}
Let
$$S= \{1,7,11,13,17,19,23,29 \},$$
 the set of integers
 which are both less than and relatively prime to $30.$ For $ x \geq 0,$ let $T_x := \{ 30x+i \; | \; i \in S\}.$ For each $ x,$  $T_x$ contains at most seven primes.
  For $s \geq 30$ let $\pi_7(s)$ denote
 the number of integers $x, \; 0 \leq x < [\frac {s}{30}],$
  for which $T_x$ contains seven primes.
  For example if $x = 0,1,2,49,62,79,89, 188,$ then $T_x$ contains seven primes. It is easy to show that $\pi_7(10^{8}) = 962.$

 In this paper we prove the following theorem which is also our main result:
\begin{thm}\label{main} For $i \geq 1$ let $ p_i $ denote the $i^{\rm th}$ prime number.
Let $m \geq 10^{10}$ be an integer and let $P_{K_m}$ denote the largest prime number less than $\sqrt{\prod_{i=1}^{m}p_i}.$ Then
$$\frac{\prod_{i=1}^{m}p_i}{8(K_m+1)} < \pi_7\left(\prod_{i=1}^{m}p_i\right). $$
\end{thm}

Since $\frac{\prod_{i=1}^{m}p_i}{8(K_m+1)}$ is an unbounded sequence, the theorem shows that there are infinitely many values of $x$ for which $T_x$ contains seven primes.
 For each $x,$ the elements of $T_x$ differ by $$2,4,6,8,10,12,16,18,22,28.$$ So we see that Theorem \ref{main} implies several cases of Alphonse de Polignac's conjecture. Since two pairs of elements of $T_x$ differ by $2,$ Theorem \ref{main} also establishes the well known twin prime conjecture.

  Our work is organized as follows: In
Section \ref{prelim} we recall the definition of the  well known sieve of Eratosthenes and record some preliminary results.
 In Section \ref{perm} we record our main observation, which is a comparison of the relative porousness of the sieve of Eratosthenes with that of another sieve that we shall refer to as the ``partition sieve."
    In Section \ref{proof2}, we prove Theorem \ref{main}. The concepts required are elementary and can be obtained from introductory texts on number theory, discrete mathematics and set theory. Some references are listed in the bibliography \cite{blair, Burt, Edgar, strayer}

 \section{Preliminary Results}\label{prelim}

 Eratosthenes ($276-194$ B.C.) was a Greek mathematician whose work in
 number theory remains significant. Consider the following lemma.

 \begin{lemma}\label{t1}
 Let $a > 1$ be an integer. If $a$ is not divisible by a prime
 number $p \leq \sqrt{a},$ then a is a prime.
 \end{lemma}

 Eratosthenes used the above lemma as a
 basis of a technique
  called ``Sieve of Eratosthenes" for finding all the prime numbers not exceeding a given
   integer $x$.
   The algorithm calls for writing down the integers from $2$ to $x$
   in their natural order. The composite numbers in the sequence are then sifted out by crossing off from $2,$ every second number (all multiples of two) in the list, from the next remaining number,
   $3,$ every third number, and so on for all the  remaining prime numbers less than or equal to
   $ \sqrt{x}.$ The integers that are left on the list  are primes. We shall refer to the set
   of integers left as the {\bf residue} of the sieve. The order of the residue set is therefore equal to $\pi(x),$ the number of primes not exceeding the integer $x.$

  We recall that $S= \{1,7,11,13,17,19,23,29 \}$ is the set of  integers which are both less than and relatively prime to $30$ and that for $ x \geq 0,$
  $T_x = \{ 30x+i \; | \; i \in S\}.$ For each integer $m \geq 1$  the sieve of Eratosthenes can be extended to the sequence of integers in the interval
  $,0 \leq x \leq m-1,$ or, equivalently, to the sets $$T_x= \{30x+1, 30x+7,30x+11, 30x+13,30x+17, 30x+19,30x+23, 30x+29\},$$
  to obtain those values of $x$ for which $T_x$ contains seven primes.
   More formally define $\phi_3(m)$ to be the number of integers $x, \; 0 \leq x \leq m-1,$ for  which $\gcd(m, 30x+i) = 1$ for all $i \in S.$
    We obtain a formula  for evaluating $\phi_3$ for certain values of $m.$

 If $p$ is a prime then $\phi_3(p)$ is easy to evaluate. For example
 $\phi_3(7) = 0$ since for all $x, \;  0 \leq x \leq 6$ the set $\{30x+i \; | \; i \in  S\}$ contains an integer divisible by $7.$ On the other hand if $p \neq  7,$ then $\phi_3(p) \neq 0.$
 It is easy to check that
  $\phi_3(p)=p$ if $p=2,3$ or $5.$ Further $\phi_3(11)=11-6$ and  $\phi_3(p)=p-8$ if $p \geq 13.$ We note also that $\phi_3(1)=1.$

 We now proceed to show that we can evaluate  $\phi_3(m)$ from the prime factorization of $m.$ Our  arguments are based on those used by Burton in \cite{Burt}, to show
  that the Euler phi-function is multiplicative.
  The following result together with its proof appear in \cite{mot}.
\begin{thm}\label{tm1} Let $k$ and $s$ be nonnegative numbers and let $p \geq 13$ be a prime number. Then:
\begin{enumerate}
\item [\rm{(i)}] $\phi_3(q^{k})= q^k$ if $q=2,3$ or $5.$ \item [\rm{(ii)}] $\phi_3(7^{s})= 0.$ \item [\rm{(iii)}]$\phi_3(11^{k})= 11^k-6\cdot11^{k-1}=11^k\left(1-\frac{6}{11}\right).$
\item [\rm{(iv)}]$\phi_3(p^{k})= p^k-8p^{k-1}=p^k\left(1-\frac{8}{p}\right).$
 \end{enumerate}
 \end{thm}
 \begin{proof}

We shall only consider the cases (iii) and
 (iv) as (i) and (ii) are easy to verify.\\[2ex]
(iii) and (iv). Clearly, for each $i \in S,$ $\gcd(30x+i, p) = 1$ if and only if $p$ does not divide $30x+i.$ Further for each $i \in S,$ there exists one integer $x$ between $0$ and $p-1$ that satisfies
the congruence relation $30x+i \equiv 0 \pmod p.$ We note however that if $p=11,$ then in $S,$ we have $23 \equiv 1 \pmod {11}$ and $29 \equiv 7 \pmod {11}.$ Hence for all $x$ for which
$30x+1 \equiv 0 \pmod {11}$ we also have $30x+23 \equiv 0 \pmod {11}$ and for all $x$ for which $30x+7 \equiv  0 \pmod {11}$ we also have $30x+29 \equiv 0 \pmod {11}.$ No such case arises when $p \geq 13.$
\end{proof}

Returning to our discussion, it follows that for each $i \in S$ there are $p^{k-1}$ integers between $0$ and $p^k-1$ that satisfy $30x+i \equiv 0 \pmod p.$
  Thus for each $i \in S,$ the set
$$\{30x+i \; | \; 0 \leq x \leq p^k-1 \}$$ contains exactly $p^k-p^{k-1}$ integers $x$ for which
   $\gcd(p^k,30x+i) = 1.$
   Since these integers $x$ are distinct for distinct elements $i \in S$ it follows that if
$p \geq 13,$ we must have $\phi_3(p^{k})= p^k-8p^{k-1}.$ However if $p=
   11$ we must have $\phi_3(11^{k})= 11^k-6\cdot11^{k-1}.$

For example $\phi_3(11^2) = 11^2 - 6\cdot11=55$ and $\phi_3(13^2) =
13^2- 8\cdot13=65.$

   In \cite{mot} it is shown that $\phi_3$ is multiplicative and that we have the following theorem.

   \begin{thm}\label{tm3} If the integer $m>1$ has the prime  factorization $$m=2^{k_{1}} 3^{k_{2}}5^{k_{3}}11^{k_{4}}p_{5}^{k_{5}}
 \cdots p_{r}^{k_{r}}$$
  with $p_s \neq 7$ for any $s \geq 5,$
  then
 \begin{eqnarray*}
 \phi_3(m)&=2^{k_{1}}3^{k_{2}}5^{k_{3}}(11^{k_{4}}-6\cdot11^{k_{4}-1})
 (p_{5}^{k_{5}}-8p_{5}^{k_{5}-1})
 \cdots (p_{r}^{k_{r}}-8p_{r}^{k_{r}-1}).\\
 \end{eqnarray*}
 \end{thm}

   We have seen that if  $p=7$ then for all $x,$ the set $\{30x+i \; | \; i \in  S\}$ contains an integer divisible by $7.$  We note further that if $p=7,$ then
    in $S,$ we have $29 \equiv 1 \pmod {7}$ and hence for all $x$ for which
$30x+1 \equiv 0 \pmod {7}$ we also have $30x+29 \equiv 0 \pmod {7}.$  Thus there exists one integer $x$ between $0$ and $6$ that simultaneously satisfies
the congruence relations $30x+1 \equiv 0 \pmod {7}$ and $30x+29 \equiv 0 \pmod {7}.$  A procedure for obtaining seven prime subsets must therefore consider the fact that $T_x$ may contain
two distinct integers divisible by $7.$
Taking this into consideration we obtain the following modification of $\phi_3.$ Let $ p_4=7, \; p_5=11, \; \ldots, \; p_{n}, \ldots $ be the ordered sequence of
 consecutive prime numbers in ascending order and let
$m := 7\cdot11\cdot13 \cdots  p_n.$ If $k \leq n,$  then by Theorem \ref{tm3}
  $$ S_7(m,k) =  6\cdot5(p_{6}-8) \cdots (p_{k}-8)p_{k+1} \cdots p_n  $$
is the number of values of $x,$ $0 \leq x \leq m-1,$ for which $T_x$ contains  $7$ integers that are relatively prime to $7\cdot11\cdot13 \cdots  p_k.$ In particular if one such value of $x$ is less than or equal to
$[\frac{p_{k+1}^2-1}{30}]$  then, by Lemma \ref{t1},
$T_x$ contains seven primes. Let
$$R_7(m,n):= \{x \; | \; 0 \leq x \leq m-1, \; T_x \;  \mbox{ contains $7$ integers that are relatively prime to $m$} \}$$
so that $|R_7(m,n)| = S_7(m,n).$ Let
 $$T_7(p^2_{n+1}) := \{x \in R_7(m,n)\; | \; x \leq [\frac{p_{n+1}^2-1}{30}] \}.$$ Then for all $x \in T_7(p^2_{n+1}),$ $T_x$ contains seven primes so that
$|T_7(p^2_{n+1})| \leq \pi_7\left(p^2_{n+1} \right).$
 We note, by writing
 \begin{eqnarray}\label{sevp1} S_7(m,n) =  m(1 - \frac{1}{7})(1 - \frac{6}{11})(1 - \frac{8}{p_6}) \cdots (1 - \frac{8}{p_n}), \end{eqnarray} that $S_7(m,n)$ may
 be computed iteratively as follows: for each $j,$ $ 4 \leq j \leq n,$ let $S_7(m,n,4) = m(1-\frac{1}{7} ),$  $S_7(m,n,5) = S_7(m,n,4)(1-\frac{6}{11} ),$ and for $j \geq 6$ let $S_7(m,n,j) = S_7(m,n,j-1)(1-\frac{8}{p_j} ).$
 Associated with the expression for $S_7(m,n)$ is a sieve on the set of integers $$0,1,2,3,4, \ldots , m-1$$
   which sifts out $x$ if either $T_x$ contains two integers divisible by $7$ or an integer $y$ not divisible by $7$ but divisible by $p_j$ for some $j,$ $5 \leq j \leq n.$ The residue set of the respective sieve is then $R_7(m,n).$
  Since the primes $p_j$ are unevenly distributed they sift out the values $x$ in an unevenly
 distributed manner. However $S_7(m,n),$ when viewed as a sieve in the manner above, is
 cyclic in the sense that when extended to the set of all integers, then for each $j\leq n,$ $p_j$ sifts out the same number of values of $x,$ with the same irregularity,
 over each interval $$s\cdot7\cdot11\cdot{\prod_{i=6}^{j} p_{i}} \leq x <   (s+1)7\cdot11\cdot{\prod_{i=6}^{j} p_{i}}.$$ The
 average density of elements of the residue set is therefore $\frac{30m}{S_7(m,n)}$
or  $$2\cdot3\cdot5\cdot{\frac{7}{6}}{\frac{11}{5}} \left( \prod_{j=6}^{n} \frac{p_j}{ (p_{j}-8)} \right).$$


   We shall require the following results of J.B. Rosser and L. Schoenfeld (see \cite{ross} page 69):
\begin{theorem}\label{ross1} {\rm[\cite{ross} Theorem 1]} Let $n \geq 1$ be an integer. Then:
\begin{itemize}
\item[{\rm (i)}] $\frac{n}{\log n -\frac{1}{2}} < \pi(n)$ \; \; for \; $n \geq 67, $
 \item[{\rm (ii)}] $ \pi(n) < \frac{n}{\log n -\frac{3}{2}} $ \; \; for $n \geq 5.$
\end{itemize}
\end{theorem}
\begin{corollary}\label{ross2}  Let $n \geq 1$ be an integer. Then:
\begin{itemize}
 \item[ ] $\frac{n}{\log n} < \pi(n) \; \;$ for \; $n \geq 17.$
 \end{itemize}
\end{corollary}
\begin{theorem}\label{ross3}  {\rm[\cite{ross} Theorem 3]} Let $n \geq 1$ be an integer. Then:
\begin{itemize}
\item[{\rm (i)}] $n(\log n +\log \log n - {\frac{3}{2}}) < p_n$ \; \; for \; $n \geq 2,$
 \item[{\rm (ii)}] $  p_n < n(\log n +\log \log n - {\frac{1}{2}})$ \; \; for \; $n \geq 20.$
\end{itemize}
\end{theorem}
\begin{corollary}\label{ross4}  Let $n \geq 1$ be an integer. Then:
 \begin{itemize}
\item[{\rm (i)}] $n(\log n ) < p_n$ \; \; for \; $n \geq 1,$
 \item[{\rm (ii)}] $  p_n < n(\log n +\log \log n )$ \; \; for \; $n \geq 6.$
\end{itemize}
\end{corollary}
As a consequence of the above results we have the following result (see \cite{moth1}):
\begin{theorem}\label{maint32} \cite{moth1}
For $n \geq 3,$ let $p_n$ denote the $n^{\rm th}$ prime. Then for each integer $b > 0$  there exists an integer $N(b)$ such that
$$ \frac {b p^2_{n+1}}{n+1} <\pi\left(p^2_{n+1} \right)$$ for all $ n \geq N(b).$
\end{theorem}

 \section{ Comparison of the sieve of Eratosthenes and the Partition sieve }\label{perm}
 We now compare the porousness of the sieve of Eratosthenes with a sieve that may be obtained from the identity
  $$1 - \sum_{s=1}^{k}\frac{1}{s(s+1)}= {\frac{1}{k+1}} .$$
  The idea is to let $x$ be a positive integer. Then we have
 \begin{equation}\label{sf1} x - \sum_{s=1}^{k}\frac{x}{s(s+1)}= {\frac{x}{k+1}}.   \end{equation}
 We then consider $\frac{x}{k+1}$ as a measure of the residue or an estimate of the number of integers that remain after applying
 $\sum_{s=1}^{k}\frac{x}{s(s+1)}$ to the sequence
 \begin{equation}\label{sf2} 1,2,3,4, \ldots , x .\end{equation} Viewed in this way we shall refer to the sieve \ref{sf1} as the {\bf partition sieve}.
 Written in this manner, the partition sieve could then be applied inductively to the sequence \ref{sf2}. We have seen, from definition, that the sieve of Eratosthenes may also be applied
  inductively for each given value of $x.$

   For each integer $s\geq 1,$ let $p_s$ denote the $s^{\rm th}$ prime number. Let $n, k $ with $n \geq k$ be a pair of integers. For each integer $x \geq p^2_{n+1}-1,$ let $S(x,k)$ denote the sum
  \begin{equation}\label{f1} S(x,k):= x + \sum_{j=1}^{k} (-1)^j \left\{ \sum_{1\leq
 s_1 < \cdots < s_{j}  \leq k }  \left[ \frac{x}{\prod_{i=1}^{j} p_{s_i}}
   \right] \right\}.
   \end{equation}
   The sum in Equation (\ref{f1}) is based on the inclusion-exclusion principle and can be considered as a sieve on the sequence of integers;
   \begin{equation}\label{f1s} 1,2,3,4,5, \ldots , x \end{equation}
   which sifts out all integers $y$ for which g.c.d.$(y, p_s) \neq 1$ for some $s, \; 1 \leq s \leq k.$
     We note that the expression for the value $S(x,k)$ sifts out the primes $p_j, \; 1 \leq j \leq k,$ from the sequence (\ref{f1s}). This is the only difference between the sieve represented by this expression and the sieve of Eratosthenes. Let $x = p^2_{n+1}-1$ for some $n \geq k.$
      We now compare the
      values $S(p^2_{n+1}-1,n)$ with $\frac{p^2_{n+1}-1}{n+1},$ the order of the residue of the sieve $p^2_{n+1}-1 - \sum_{s=1}^n \frac{p^2_{n+1}-1 }{s(s+1)}.$
     The comparison can be achieved inductively. Let ${\cal S}(x,k),$ denote the set of all positive integers not exceeding $x$ which are relatively prime to the primes $p_j, \; 1 \leq j \leq k.$ Then $ S(x,k) = |{\cal S}(x,k)|.$ Note that for $x = p^2_{n+1}-1,$ the effect of $S(x,n)$ on the sequence \ref{f1s} coincides with that of the sieve of Eratosthenes apart from the fact that $S(x,n)$ also sifts out the primes $2,3,5, \ldots, p_n.$ Thus  $\pi\left(p^2_{n+1}\right) =  S(p^2_{n+1}-1,n)+n-1.$

    In the following result we show, in particular, that the number of primes between $p_{n}$ and $p^2_{n+1}$ is unbounded as $n$ increases. The result is an immediate consequence of Theorem \ref{maint32} (see \cite{moth1}). In fact the rest of the results in this section and their proofs are mere extensions of our work \cite{moth1}.

   \begin{corollary}\label{th3} For $n \geq 12,$ let $p_n$ denote the $n^{\rm th}$ prime. Then for each integer $d \geq 1$  there exists an integer $N(d)$ such that
   $\frac {d p^2_{n+1}}{n+1} < S( p^2_{n+1}-1,n),$ for all $ n \geq N(d).$
   \end{corollary}
\begin{proof}
   If $n \geq 12,$ then
   $ \frac {b p^2_{n+1}}{n+1} <\pi\left(p^2_{n+1} \right)$ for some integer $b \geq 2.$
   Since $S( p^2_{n+1}-1,n) = \pi\left(p^2_{n+1} \right ) - n+1$ and
   $ n-1 < \frac {(n+1)^2 \log^2 (n+1)}{n+1} < \frac { p^2_{n+1}}{n+1},$
 we have
 $$  \frac {(b-1) p^2_{n+1}}{n+1} < \frac {b p^2_{n+1}}{n+1} -n+1 < \pi\left(p^2_{n+1} \right) - n+1  = S( p^2_{n+1}-1,n)$$ for all  $n \geq 12.$
  The result of the corollary follows if we put $d=b-1.$
 \end{proof}

 Since $\frac {d (p^2_{n+1}-1)}{n+1} < \frac {d p^2_{n+1}}{n+1}$ we see that if for each integer $n \geq 12,$ we put $d_n = \frac {(n+1)S(p^2_{n+1}-1,n)}{p^2_{n+1}-1}, $ then, from the result of Corollary \ref{th3}, we have an unbounded sequence of rational numbers $\{d_n\}_{n \geq 12}$ such that
$$S(p^2_{n+1}-1,n) = \frac {d_n (p^2_{n+1}-1)}{n+1}.$$ But for each $ n \geq 4,$ $S(p^2_{n+1}-1,n)$ may be computed inductively from $S( p^2_{n+1}-1,3),$ forming a finite sequence of values $S( p^2_{n+1}-1,k),$ $3 \leq k \leq n.$ For each $n \geq 4 $ and $k, \; 3 \leq k \leq n,$ we have \\
 $S(p^2_{n+1}-1,k+1) = S(p^2_{n+1}-1,k) - T( p^2_{n+1}-1,k+1),$
     where
     \begin{equation}\label{fse16}T( p^2_{n+1}-1,k+1) :=  \left[ \frac{ p^2_{n+1}-1}{p_{k+1}} \right] + \sum_{j=1}^{k} (-1)^{j} \left\{ \sum_{1\leq
 s_1 < \cdots < s_{j}  \leq k }  \left[ \frac{ p^2_{n+1}-1}{{p_{k+1}}\prod_{i=1}^{j} p_{s_i}}
   \right] \right\} . \end{equation}
In the same vain, $d_n$ is the last term of a sequence of numbers $\{a_k(n)\}, \; 3 \leq k \leq n,$ defined, for each fixed integer $n \geq k,$ by $a_k(n) :=  \frac {(k+1)S(p^2_{n+1}-1,k)}{p^2_{n+1}-1}.$

The following is our main observation in this section (see \cite{moth1}):

\begin{lemma}\label{th30} Let $n,$ $k$ with, $ k \leq n,$ be a pair of positive integers. Then
$8T(p^2_{n+1}-1,k+1) < \frac {a_{k}(n)( p^2_{n+1}-1)}{(k+1)(k+2)}$ for all $k \geq 10^{10}.$
\end{lemma}
\begin{proof} We first
%
   show that for each fixed $n$ and all $k \leq n,$
 $\{a_{k}(n)\}$ is a nondecreasing sequence.
 To get a more explicit estimate for $a_{k}(n)$ for a given value of $k \geq 30,$ we note that if $k=n,$ then
  \begin{eqnarray}\label{fv9}
 a_{n}(n)&= &\frac {(n+1)S(p^2_{n+1}-1,n)}{p^2_{n+1}-1} \nonumber \\[2ex]
 \hspace{4ex} & = & \frac {(n+1)(\pi\left(p^2_{n+1} \right ) - (n-1) )}{p^2_{n+1}-1} \nonumber \\[2ex]
 \hspace{4ex} & > & \frac {(n+1)(\frac{ p_{n+1}^2 }{(2{\log}(p_{n+1}))} - (n-1) )}{p^2_{n+1}-1} \nonumber \\[2ex]
 \hspace{4ex} & > & \frac {(n+1)}{(2{\log}(p_{n+1}))} - \frac{(n+1) (n-1)}{(n+1)^2 ({\log}(n+1))^2} \; \; \;  > \; \;  \frac {(n+1)}{(2{\log}(p_{n+1}))} - \frac{1}{10}. \nonumber
 \end{eqnarray}
 By Theorem \ref{ross1}, $\frac {m}{({\log}(\frac{m}{1.64}))} < \frac {m}{({\log}(m-\frac{1}{2}))} < \pi (m)$ for $m \geq 67.$ We see therefore that we must have \\
 $a_{n}(n) > \frac {(n+1)}{(2{\log}(p_{n+1}))}$ for all $n \geq 30.$

 Recall that ${\cal S}(p^2_{n+1}-1,k)$ represents the residue set of the sieve of Equation \ref{f1}, that is, is the set of all positive integers not exceeding $p^2_{n+1}-1$ which are relatively prime to
 the primes $p_s, \; 1 \leq s \leq k.$ Let $ \{q^{k}_r\}, \; r \geq 1$ be the ordered sequence of elements of ${\cal S}(p^2_{n+1}-1 ,k),$ so that $q^{k}_1 = 1, q^{k}_2= p_{k+1}, q^{k}_3 = p_{k+2}, \ldots .$ Then
  $q^{k}_r$ is a prime whenever $q^{k}_r < p^2_{k+1}.$ For $n > k,$  ${\cal S}(p^2_{n+1}-1,k+1)$  is obtained from ${\cal S}(p^2_{n+1}-1, k)$ by sifting out all products $q^{k}_r p_{k+1}$ less than $p^2_{n+1}-1,$ where, for each $r,$ $q^{k}_r$ is an element of ${\cal S}(p^2_{n+1}-1, k)$ or, equivalently,
   $$S(p^2_{n+1}-1,k+1) = S(p^2_{n+1}-1,k) - T( p^2_{n+1}-1,k+1).$$
  Now
  $$\frac {a_k(n)(p^2_{n+1}-1)}{k+1} = S(p^2_{n+1}-1,k)$$ and
  \begin{eqnarray}\label{fce9}
 \frac {a_{k}(n)(p^2_{n+1}-1)}{k+2}& = &\frac {a_{k}(n)(p^2_{n+1}-1)}{k+1} - \frac{a_{k}(n)(p^2_{n+1}-1) }{(k+1)(k+2)} \nonumber \\[2ex]
\hspace{4ex} & =  &\frac {p^2_{n+1}-1}{\frac{1}{a_{k}(n)}(k+1)} - \frac{p^2_{n+1}-1 }{\frac{1}{a_{k}(n)} (k+1)(k+2)} \nonumber
 \end{eqnarray}
 Thus $a_{k+1}(n) = a_k(n)$ if $$T( p^2_{n+1}-1,k+1) = \frac{p^2_{n+1}-1 }{\frac{1}{a_{k}(n)} (k+1)(k+2)}.$$ It follows that $a_{k+1}(n) \geq a_k(n)$ only if
 $$T( p^2_{n+1}-1,k+1) \leq \frac{p^2_{n+1}-1 }{\frac{1}{a_{k}(n)} (k+1)(k+2)}.$$
  From our remarks above, it suffices to show that
 \begin{equation}\label{fs2}
 r(\frac{1}{a_k(n)}(k+1)(k+2))< p_{k+1}q^{k}_r
  \end{equation}
 for each $r \geq 1$ for which both products are less than $p^2_{n+1}-1.$ Since $(k+1)(\log (k+1) ) < p_{k+1},$  it suffices to show that $r(\frac{1}{a_k(n)}(k+2)) < (\log (k+1) )q^{k}_r$ or, equivalently,
 $\frac{r}{\log (k+1)}(\frac{1}{a_k(n)}(k+2)) < q^{k}_r.$ If $1<q^{k}_r <  p^2_{k+1},$ then
 $ q^{k}_r$ is equal to a prime number $p_{s}$ with $s >k.$ We know that $s\log s < p_s.$ Treating $s\log s$ as a function of $s$ we get its derivative to be $1+\log s.$ Treating
  $\frac{r}{\log (k+1)}(\frac{1}{a_k(n)}(k+2))$ as a function of $r$ we get its derivative to be equal to $\frac{1}{\log (k+1)}(\frac{1}{a_k(n)}(k+2)).$ By virtue of our estimation of $a_n(n)$ above we assume
   that $a_k(n) \geq \frac {(k+1)}{(2{\log}(p_{k+1}))}$ or that $ \frac {(k+1)}{(2{\log}(p_{k+1}))}$ is a close estimate for $a_k(n).$
   Then we would have that
  $\frac{1}{\log (k+1)}(\frac{1}{a_k(n)}(k+2))$ is less than or approximately equal to  $\frac {(k+2)(2{\log}(p_{k+1}))}{(k+1)(\log (k+1) )}.$
  Now
   \begin{eqnarray}\label{fce9}
 \frac {(k+2)(2{\log}(p_{k+1}))}{(k+1)(\log (k+1))}& = &\frac {2{\log}(p_{k+1})}{\log (k+1) } + \frac {2{\log}(p_{k+1})}{(k+1)\log (k+1) } \nonumber \\[2ex]
\hspace{4ex} & <  &\frac {2{\log}((k+1)(\log (k+1)+ \log \log (k+1) ))}{\log (k+1) } + 0.1 \nonumber \\[2ex]
\hspace{4ex} & =  & 2 + \frac {2{\log}(\log (k+1)+ \log \log (k+1) )}{\log (k+1) } + 0.1  \; \; \; \; < 3 \nonumber
 \end{eqnarray}
 But
  $3 < 1+\log s, \; \; s \geq k+1$ and $k \geq 30.$ This establishes the Inequality (\ref{fs2}) for
   $1\leq q^{k}_r <  p^2_{k+1}.$ For $k \geq 30,$ the set ${\cal S}(p^2_{n+1}-1,k)$ is more dense over the interval $1\leq q^{k}_r <  p^2_{k+1},$ than over the interval
   $q^{k}_r \geq  p^2_{k+1}.$ The above argument therefore suffices for the cases $q^{k}_r \geq  p^2_{k+1}.$
   For $s \geq 10^{10},$ $24 = 8\cdot3 < 1+\log s$ and
     this completes the proof of the lemma.
  \end{proof}
  Thus when $k \geq 10^{10},$ then at each inductive stage, the partition sieve is more porous (leaves a smaller residue set) by at least eight times the porousness of the sieve of Eratosthenes or that of Equation \ref{f1}. In particular, the result of Lemma \ref{th30} is not dependent on $x$ being equal to $p^2_{n+1}-1,$ but could be extended to any value of $x \geq  p^2_{n+1}.$

  \section{Proof of Theorem \ref{main}}\label{proof2}

 The result of Theorem \ref{main} shall be seen to be a consequence of the following observation.

 Let $p_i$ denote the $i^{\rm th}$ prime number and let $m \geq T=10^{10}.$ Let  $P_{K_m}$ denote the largest prime number less than
 $$\sqrt{\prod_{i=1}^{m}p_i}.$$
  By the result of Lemma \ref{th30}, it suffices to show that
\begin{eqnarray}\label{iq1}\frac{\prod_{i=1}^{m}p_i}{8(T+1)} < S_7\left(\prod_{i=4}^{m}p_i , T \right) = 6\cdot5\prod_{i=6}^{T}(p_i-8)\prod_{i=T+1}^{m}p_i \end{eqnarray}
or, equivalently, that
 $$ \frac{8\cdot (T+1)}{2\cdot3\cdot5}\cdot{\frac{6}{7}}\cdot{\frac{5}{11}} \left( \prod_{j=6}^{T} \frac{(p_j -8)}{ p_{j}} \right)\prod_{i=T+1}^{m}p_i >1 .$$
 This can be shown to be the case by direct computation when $m=T.$ This, in turn, implies that the inequality \ref{iq1} holds
 for all $m \geq 10^{10}.$

 Now for each $m$ apply the sieve \ref{sf1} inductively to $\frac{\prod_{i=1}^{m}p_i}{8(T+1)}$ starting with $k = T.$ In practice, for each fixed $m$ we apply the sieve
 $$\frac{\prod_{i=1}^{m}p_i}{8(T+1)} - \sum_{r=T+1}^{k}\frac{\prod_{i=1}^{m}p_i}{8r(r+1)}= {\frac{\prod_{i=1}^{m}p_i}{8(k+1)}} ,$$ inductively, replacing $\frac{\prod_{i=1}^{m}p_i}{8(T+1)}$ by its respective residue at each stage and letting $T < k \leq K_m.$

  On the other hand there is no simply expression for a sieve that can be applied to $\prod_{i=4}^{m}p_i $ to yield a residue set of order $\pi_7\left(\prod_{i=1}^{m}p_i \right).$ Instead, for each $k,$ $T < k \leq K_m,$
  and, as explained in Chapter \ref{prelim},
  consider the sieve associated with the expression
 $$S_7(\prod_{i=4}^{K_m}p_i, k ) = 6\cdot5\prod_{i=6}^{k}(p_i-8)\prod_{i=k+1}^{K_m}p_i   .$$ Then the residue set for the sieve will be
  $$R_7(\prod_{i=4}^{K_m}p_i,k) = \{ x \; | \; 0 \leq x < \prod_{i=4}^{K_m}p_i, \; T_x \; \mbox{ contains $7$ integers that are relatively prime to $\prod_{i=4}^{k}p_i$} \}. $$  Now let $k = K_m$ and
  let
  $$T_7(\prod_{i=1}^{m}p_i)  = \{ x  \in R_7(\prod_{i=4}^{K_m}p_i,k) \; | \;  x < \prod_{i=4}^{m}p_i \;\}.$$
 By Lemma \ref{t1} if $x \in T_7(\prod_{i=1}^{m}p_i),$ then $T_x$ contains seven primes so that
 $$|T_7(\prod_{i=1}^{m}p_i)| < \pi_7\left(\prod_{i=1}^{m}p_i \right).$$
 Now let $k , T < k \leq K_m$ be given. By  Lemma \ref{th30}, for each value of $x$ for which $T_x$ contains a multiple of
of a prime $p_j$ ($10^{10}<j \leq K_m$ and with $p_j$ as its smallest divisor),
 there corresponds at least eight integers $r_i$ (unique to $k$) for which each satisfies
 $$r_i({\frac{1}{a_{k}(n)} (k+1)(k+2)})< \prod_{i=1}^{m}p_i$$
 and bears the relation \ref{fs2} with $ p_j.$ There therefore exists an integer $r,$ determined by the integers $r_i,$ (hence unique to $k$) such that
 $$r({\frac{8}{a_{k}(n)} (k+1)(k+2)})< \prod_{i=1}^{m}p_i.$$

 Thus in proceeding from $k=T$ to $k = K_m$ the Sieve \ref{sf1} sifts out more elements from the initial residue set of order
 $\frac{\prod_{i=1}^{m}p_i}{8(T+1)}$ than the extension of the Sieve  \ref{f1} from it initial residue set of order $$6\cdot5\prod_{i=6}^{T}(p_i-8)\prod_{i=T+1}^{m}p_i.$$ Note that $|T_x| = 8$ and dividing by ${\frac{8}{a_{k}(n)} (k+1)(k+2)}$ sifts out eight more integers than dividing by ${\frac{1}{a_{k}(n)} (k+1)(k+2)}.$
 %
  But
  $$\frac{\prod_{i=1}^{m}p_i}{8(T+1)} <  6\cdot5\prod_{i=6}^{T}(p_i-8)\prod_{i=T+1}^{m}p_i$$
   so, by the foregoing, the orders of the final residue sets must bear the relation:
   $$ \frac{\prod_{i=1}^{m}p_i}{8\cdot (K_m+1)} <  |T_7(\prod_{i=1}^{m}p_i)| <  \pi_7\left(\prod_{i=1}^{m}p_i\right). $$
   This completes the proof of the theorem.

\begin{intere}
 The authors declare that they do not have any competing interest.
 \end{intere}
 \begin{acknow} An earlier version of this paper is presented as a preprint on arxiv.org. and may be viewed at the following link:
  https://arxiv.org/abs/1912.09290"
  \end{acknow}

\end{document}